\documentclass[12pt]{amsart}
\usepackage{amsmath}
\usepackage{amssymb}
\usepackage[mathcal]{eucal}
\usepackage[all]{xy}
\usepackage{latexsym}
\usepackage{amstext}
\usepackage{amsfonts}
\usepackage{amsthm}
\usepackage{amsopn}
\usepackage{amsbsy}
\usepackage{layout}
\usepackage{color}
\usepackage{graphicx}
\usepackage{bm}
\usepackage{yfonts}
\usepackage{pifont}
\usepackage[scr]{rsfso}
\usepackage{tikz}
\newtheorem{thm}{Theorem}[section]
\newtheorem{lem}[thm]{Lemma}

\newtheorem{cor}[thm]{Corollary}
\newtheorem{pro}[thm]{Proposition}

\newtheorem{claim}{Claim}

\newcommand{\Int}{\mbox{{\rm Int}}}

\def\supp{{\mathrm {supp}}\,}

\makeatletter
\def\Int{\mathop{\operator@font Int}\nolimits}
\makeatother

\begin{document}

\title[Linear continuous operators with bounded supports]
{Linear continuous operators with bounded supports}

\author{Vesko Valov}
\address{Department of Computer Science and Mathematics, Nipissing University,
100 College Drive, P.O. Box 5002, North Bay, ON, P1B 8L7, Canada}
\email{veskov@nipissingu.ca}
\thanks{The author was partially supported by NSERC Grant RGPIN-2025-07173}



\begin{abstract}
For any Tychonoff space $X$ let $D(X)$ be either the set $C(X)$ of all continuous functions on $X$ or the set $C^*(X)$ of all bounded continuous functions on $X$. When $D(X)$ is endowed with the point convergence topology, we write $D_p(X)$.  Zakrzewski \cite[Theorem 3.12]{kz} proved that if $X$ and $Y$ are $\sigma$-compact spaces and
there is a continuous linear map $T:C_p(X)\to C_p(Y)$ such that $T(C_p(X))$ is dense in $C_p(Y)$ and $|\supp(y)|\leq m$ for every $y\in Y$, then
$\dim Y\leq m\cdot\dim X+m+m!-1$. Here, $\supp(y)$ denotes the support of the linear continuous map $l_y:C_p(X)\to\mathbb R$, defined by 
$l_y(f)=T(f)(y)$. In the present paper we improve the last inequality by showing that $\dim Y\leq m\cdot\dim X$ provided $X,Y$ are Tychonoff spaces and there is a continuous linear surjection $T:D_p(X)\to D_p(Y)$ with $|\supp(y)|\leq m$ for every $y\in Y$. The methods developed to prove this result yields a generalization of \cite[Theorem 1.4]{ev}: If  $T:D_p(X)\to D_p(Y)$ is a continuous linear surjection with $X,Y$ Tychonoff spaces and $\dim X=0$, then $\dim Y=0$. 
\end{abstract}

\maketitle\markboth{}{Operators with bounded supports}




\section{Introduction}\label{intro}
For a Tychonoff space $X$
 we denote by $C(X)$ the linear space of all continuous real-valued functions on $X$.
$C^*(X)$ is the subspace of $C(X)$ consisting of the bounded functions.
Everywhere below, by $D(X)$ we denote either $C^*(X)$ or $C(X)$, and $D_p(X)$ stays for $D(X)$ endowed with the point-wise convergence topology.  
If $T:D_p(X)\to D_p(Y)$ is a continuous linear map, then there are four possible cases: $D(X)$ is either $C(X)$ or $C^*(X)$ and $D(Y)$ is either $C(Y)$ or $C^*(Y)$. We write $D_p(X)$ (resp., $D_p(Y)$) if $D(X)$ (resp., $D(Y)$) is endowed with the point-wise convergence topology.
More information about function spaces with the point-wise convergence topology can be found in 
\cite{ar1}, \cite{vanMill}, \cite{tk}.

Throughout the paper by dimension we mean the {\em covering dimension $\dim$}. 
Recall that for a Tychonoff space $X$ and an integer $n\geq 0$, $\dim X \leq n$ if every finite functionally open cover of the space $X$
 has a finite functionally open refinement of order $\leq n$, see \cite{en}.
According to that definition, we have $\dim X=\dim\beta X$, where $\beta X$ is the \v{C}ech-Stone compactification of $X$.
After the striking results of Pestov \cite{p} and Gul'ko \cite{gu} that $\dim X=\dim Y$ for any Tychonoff spaces $X$ and $Y$ 
provided $C_p(X)$ and $C_p(Y)$ are linearly homeomorphic or uniformly homeomorphic, Arhangel'skii \cite{ar} posed the problem
whether $\dim Y\leq\dim X$ if there is a continuous linear surjection from $C_p(X)$ onto $C_p(Y)$.
This question was answered negatively by Leiderman-Levin-Pestov \cite{llp} and Leiderman-Morris-Pestov \cite{lmp}: 
For every finite-dimensional metrizable compact space $Y$ there exists a continuous linear surjection $T: C_p([0,1])\to C_p(Y)$ \cite{lmp}. 

However, it turned out that the zero-dimensional case is an exception.  It was shown in \cite{llp} that if there is a linear continuous surjection 
$T: C_p(X)\to C_p(Y)$ for compact metrizable spaces $X$ and $Y$,
 then $\dim X=0$ implies that $\dim Y=0$. The last result was extended for arbitrary compact spaces by Kawamura-Leiderman \cite{KawL}
 who raised the question if their result holds for arbitrary Tychonoff spaces $X$ and $Y$.
 Recently, this difficult question was answered positively in \cite{ev}. In the present paper we show this is also true for linear continuous surjections
$T:D_p(X)\to D_p(X)$, see Theorem 1.3 below.  

Theorem 1.3 is obtained as a result of the methods developed to extend the following theorem of
Zakrzewski \cite[Theorem 3.12]{kz}: If $X$ and $Y$ are $\sigma$-compact spaces 
and there is a continuous linear map $T:C_p(X)\to C_p(Y)$ such that $T(C_p(X))$ is dense in $C_p(Y)$ and $|\supp(y)|\leq m$ for every $y\in Y$, then
$\dim Y\leq m\cdot\dim X+m+m!-1$. Here, $\supp(y)$ denotes the support of the linear continuous map $l_y:C_p(X)\to\mathbb R$, defined by 
$l_y(f)=T(f)(y)$, see the precise definition below. In the present paper we improve the last inequality.

\begin{thm}\label{theorem-main} Let $X$ and $Y$ be Tychonoff spaces and $T:D_p(X)\to D_p(Y)$ be a surjective continuous linear map such that $|\supp(y)|\leq m$ for every $y\in Y$. Then $\dim Y\leq m\cdot\dim X$.
 \end{thm} 
Note that in case both $X$ and $Y$ are metrizable, Theorem 1.1 was established in \cite{lv}.
\begin{cor}
Suppose $X,Y$ are normal spaces and there is a linear continuous surjection $T:D_p(X)\to D_p(Y)$. If $X$  is strongly countable-dimensional, then so is $Y$.
\end{cor} 
When both $X$ and $Y$ are compact metrizable spaces, Corollary 1.2 was established in \cite[Theorem 4.2(b)]{gkm}.

The methods from the proof of Theorem 1.1 provide also the next result (it was established in \cite{ev} when $T$ is either a surjection between $C_p(X)$ and $C_p(Y)$ or between $C_p^*(X)$ and $C_p^*(Y)$).
\begin{thm}
If $X$ and $Y$ are Tychonoff spaces and  $T:D_p(X)\to D_p(Y)$ is a surjective continuous linear map such that $\dim X=0$, then $\dim Y=0$.
\end{thm} 
A few words about the structure of the paper. In Section 2 we provide the main properties of the supports of linear continuous maps when the domains are $QS$-algebras. Section 3 contains some technical lemmas, which are used in the proof of Theorem 1.1. Section 4 contains all proofs. Specially, the proof of Theorem 1.1 is reduced to the case when both $X$ and $Y$ are separable metric spaces.  

\section{$QS$-algebras and supports of linear functionals}
Let $\mathbb Q$ be the set of rational numbers. A subspace $E(X)\subset D(X)$ is called a {\em $QS$-algebra} \cite{gu} if it satisfies the following conditions: (i) If $f,g\in E(X)$ and $\lambda\in\mathbb Q$, then all functions $f+g$, $f\cdot g$ and $\lambda f$ belong to $E(X)$; (ii) For every $x\in X$ and its neighborhood $U$ in $X$ there is $f\in E(X)$ such that $f(x)=1$ and $f(X\backslash U)=0$. 

We need the following facts from \cite{gu}: 
\begin{itemize}
\item[(2.1)] If $X$ has a countable base and $\Phi\subset D(X)$ is a countable set, then there is a countable $QS$-algebra $E(X)\subset D(X)$ containing $\Phi$; 
\item[(2.2)] If $E(X)$ is a $QS$-algebra on $X$ and $U\subset X$ is an open set containing the points $x_1,x_2,..,x_k$ and $\lambda_1,\lambda_2,..,\lambda_k\in\mathbb Q$, then there exists $f\in E(X)$ such that $f(x_i)=\lambda_i$ for each $i$ and $f(X\backslash U)=0$.
     
 We also need the following condition for $QS$-algebras:     
\item[(2.3)] We consider the following condition for a $QS$-algebra $E(X)$ on $X$: For every compact set  $K\subset X$ and an open set $W$ containing $K$ there exists $f\in E(X)$ with $f|K=1$, $f|(X\backslash W)=0$ and $f(x)\in [0,1]$ for all $x\in X$. Note that 
if $X$ has a countable base $\mathcal B$, then there is a countable $QS$-algebra $E(X)$ on $X$ satisfying that condition. Indeed, we can assume that $\mathcal B$ is closed under finite unions and find $U,V\in\mathcal B$ such that $K\subset V\subset cl(V)\subset U\subset cl(U)\subset W$. Then consider the set $\Phi$ of all functions $f_{U,V}:X\to [0,1]$, where $cl(V)\subset U$ with $U,V\in\mathcal B$, such that  
    $f_{V,U}|cl(V)=1$ and $f_{V,U}|(X\backslash U)=0$. According to $(2.1)$, $\Phi$ can be extended to a countable $QS$-algebra $E(X)$ on $X$. 
\end{itemize}
Suppose $\overline X$ and $\overline Y$ are compactifications of the spaces $X$ and $Y$, respectively. Let $E(X)\subset D(X)$ and $E(Y)\subset D(Y)$ be $QS$-algebras on $X$ and $Y$ such that every $f\in LE(X)$ and $g\in LE(Y)$ can be extended to maps $\overline f\in C(\overline X,\overline{\mathbb R})$ and $\overline g\in C(\overline Y,\overline{\mathbb R})$, where $EL(X)$ and $EL(Y)$ are the linear hulls of $E(X)$ and $E(Y)$, respectively. Suppose also that there is a continuous surjection 
$\varphi_0:E(X)\to E(Y)$ such that $\varphi_0(f_1\pm f_2)=\varphi_0(f_1)\pm\varphi_0(f_2)$ for all $f_1,f_2\in E(X)$ (such a map $\varphi_0$ is called semi-linear). Let $\varphi:LE_p(X)\to LE_p(Y)$ be a linear continuous map extending $\varphi_0$.
For every $y\in\overline Y$ we define the {\em support of $y$} to be the set $\supp(y)$ of all $x\in\overline X$ satisfying the following condition: for every neighborhood $U\subset\overline X$ of $x$ there is $f\in LE(X)$ such that $f(X\backslash U)=0$ and $\overline{\varphi(f)}(y)\neq 0$, where  
$\overline{\varphi(f)}\in C(\overline Y,\overline{\mathbb R})$ is the extension of $\varphi(f)$. Our definition of supports is a modification of that notion given in \cite{ar1} and\cite{vv}.
Obviously, each $\supp(y)$ is a closed subset of $\overline X$.

The next proposition is an analogue of \cite[Proposition 4.1]{ev}.
\begin{pro}
Suppose $\overline X$ and $\overline Y$ are compactifications of the spaces $X$ and $Y$, respectively, and $E(X)\subset D(X)$, $E(Y)\subset D(Y)$ are $QS$-algebras on $X$ and $Y$ such that every $f\in LE(X)$ and $g\in LE(Y)$ can be extended to maps $\overline f\in C(\overline X,\overline{\mathbb R})$ and $\overline g\in C(\overline Y,\overline{\mathbb R})$. Suppose the following conditions hold:
\begin{itemize}
\item[(1)] For every finite open cover $\gamma=\{U_i\}_{i=1}^k$ of $\overline X$ there is a partition of unity $\{\overline f_i\}_{i=1}^k$ subordinated to $\gamma$ with $f_i\in E(X)$;
\item[(2)] The family $E(\overline X)=\{\overline f:f\in E(X)\}$ contains a $QS$-algebra of function from $C(\overline X)$ satisfying condition $(2.3)$;
\item[(3)] The real-valued elements of $E(\overline Y)=\{\overline g:g\in E(Y)\}$ is dense in $C_p(\overline Y)$;     
\end{itemize}
If there is a continuous semi-linear surjection $\varphi_0:E(X)\to E(Y)$ which can be continuously extended to a linear map $\varphi:LE_p(X)\to LE_p(Y)$ between the linear hulls of $E(X)$ and $E(Y)$, then we have:
\begin{itemize}
\item[(a)] If $U\subset\overline X$ is open and contains $\supp(y)$, then $\overline{\varphi(f)}(y)=0$ for every $f\in LE(X)$ with $f(U\cap X)=0$;
\item[(b)] For every $y\in\overline Y$ the set $\supp(y)$ is a non-empty closed subset of $\overline X$ and $\supp(y)$ is a finite subset of  $X$ provided $y\in Y$;
\item[(c)] The set-valued map $y\rightsquigarrow\supp(y)$ is lower semi-continuous, i.e. the set $\supp^{-1}(U)=\{y\in\overline Y:\supp(y)\cap U\neq\varnothing\}$ is open in $\overline Y$ for every open $U\subset\overline X$;
\item[(d)] For any $p,k$ the sets $\overline Y_{p,k}=\{y\in\overline Y:|\supp(y)|\leq k{~}\hbox{and}{~}a(y)\leq p\}$ and $\overline Y_k=\{y\in\overline Y:|\supp(y)|\leq k\}$
are closed in $\overline Y$ such that $Y\subset\bigcup_{p,k\geq 1}\overline Y_{p,k}$, where
$$a(y)=\sup\{|\overline{\varphi(f)}(y)|:f\in LE(X){~}\hbox{and}{~}|\overline f(x)|<1{~} \forall x\in supp(y)\}.$$ 
\item[(e)] If $y\in cl(\overline Y_{p,k}\cap Y)\backslash\overline Y_{k-1}$,  then there exist unique real numbers $\lambda_i(y)$, $i=1,2,..,k$, such that $\overline{\varphi(f)}(y)=\sum_{i=1}^k\lambda_i(y)\overline f(x_i(y))$ for all $f\in LE(X)$ and $\sum_{i=1}^k|\lambda_i(y)|=a(y)$, where $\supp(y)=\{x_1(y),.,x_k(y)\}$. Moreover, the functions $\lambda_i$ are continuous on the set $A(y)=\{z\in cl(\overline Y_{p,k}\cap Y)\backslash\overline Y_{k-1}:\supp(z)=\supp(y)\}$.
\end{itemize}
\end{pro}
\begin{proof}
$(a)$ Let $U\subset\overline X$ be an open set containing $\supp(y)$ for some $y\in\overline Y$ and $f(U\cap X)=0$ with $f\in LE(X)$. Since $\supp(y)$ is closed in $\overline X$, we can assume that $U$ is a finite union of open sets $V_i$, $i=1,2,..,k$.
Every $x\in\overline X\backslash U$ has a neighborhood $V_x$ such that $\overline{\varphi(g)}(y)=0$ for any $g\in LE(X)$ with $g(X\backslash V_x)=0$. 
Take a finite open cover $\{V_{x_1},..,V_{x_m}\}$ of $\overline X\backslash U$. Then 
$\gamma=\{V_1,..,V_k,V_{x_1},..,V_{x_m}\}$ is an open cover of $\overline X$ and there exits a partition of unity $\{\overline h_1,.,\overline h_k,\overline\theta _1,..,\overline\theta_m\}\subset E(\overline X)$ subordinated to $\gamma$. Hence, $h_i\cdot f, \theta_j\cdot f\in LE(X)$ for all $i,j$ and
$f=\sum_{i=1}^kh_i\cdot f+\sum_{j=1}^m\theta_j\cdot f$. Take a net $\{y_\alpha\}\subset Y$ with $\lim y_\alpha=y$. 
So, $\varphi(f)(y_\alpha)=\sum_{i=1}^k\varphi(h_i\cdot f)(y_\alpha)+\sum_{j=1}^m\varphi(\theta_j\cdot f)(y_\alpha)$.
Observe that $(h_i\cdot f)(x)=0$ for all $x\in X$, so $\varphi(h_i\cdot f)$ is the zero function on $Y$ and $\varphi(h_i\cdot f)(y_\alpha)=0$ for all $\alpha$ and $i=1,..,k$. 
Consequently, $\varphi(f)(y_\alpha)=\sum_{j=1}^m\varphi(\theta_j\cdot f)(y_\alpha)$.
On the other hand, $(\theta_j\cdot f)(x)=0$ for all $x\in X\backslash V_{x_j}$. 
So, $\overline{\varphi(\theta_j\cdot f)}(y)=0$. Since $\lim_\alpha\varphi(\theta_j\cdot f)(y_\alpha)=\overline{\varphi(\theta_j\cdot f)}(y)$ for all $j=1,..,m$ and 
$\lim_\alpha\varphi(f)(y_\alpha)=\overline{\varphi(f)}(y)$, we have $\overline{\varphi(f)}(y)=0$. This easily implies that $\overline{\varphi(f)}(y)=\overline{\varphi(g)}(y)$ for any $f,g\in LE(X)$ with $f|(U\cap X)=g|(U\cap X)$.  

$(b)$ Take $y\in\overline Y$ and a function $g\in E(Y)$ with $\overline g(y)\neq 0$ (this can be done
because the set of real-valued functions from $E(\overline Y)$ is dense in $C_p(\overline Y))$.
Choose $f_y\in E(X)$ such that $\varphi(f_y)=g$, so $\overline{\varphi(f_y)}(y)=\overline g(y)$. If $supp(y)=\varnothing$, then every $x\in\overline X$ has a neighborhood $U_x$ such that $\overline{\varphi(f)}(y)=0$ for any $f\in LE(X)$ with $f(X\backslash U_x)=0$. Choose a finitely cover 
$\gamma=\{U_{x_i}:i=1,2,..,k\}$ of $\overline X$ and a partition of unity $\{\overline h_i:i\leq k\}$ subordinated to $\gamma$. Hence, $f_y=\sum_{i=1}^kh_i\cdot f_y$ and $\varphi(f_y)(y_\alpha)=\sum_{i=1}^k\varphi(h_i\cdot f_y)(y_\alpha)$ for each $\alpha$, where  $\{y_\alpha\}\subset Y$ is a net converging to $y$.
Since $\lim_\alpha\varphi(h_i\cdot f_y)(y_\alpha)=\overline{\varphi(h_i\cdot f_y)}(y)=0$ for every $i$, we have 
$\overline{\varphi(f_y)}(y)=\sum_{i=1}^k\lim_\alpha\varphi(h_i\cdot f_y)(y_\alpha)=0$, a contradiction.

Because $\varphi$ is continuous, for every
$y\in Y$ the equality $l_y(f)=\varphi(f)(y)$ defines a continuous linear functional $l_y:LE_p(X)\to\mathbb R$. Thus, there are $\varepsilon >0$ and a finite set $K=\{x_1(y),x_2(y),..,x_k(y)\}\subset X$ such that $|l_y(f)|<1$ for all $f\in LE(X)$ with $|f(x)|<\varepsilon$, $x\in K$.
Using the linearity of $l_y$, one can show that $l_y(g)=0$ for all $g\in LE(X)$ with $g(x_i(y))=0$ for all $i$. 
Since $E(X)$ is a $QS$-algebra, for every $i$ there is a function $g_i\in E(X)$ such that $g_i(x_i(y))=1$ and $g_i(x_j(y))=0$ with $j\neq i$. Now, if 
$f\in LE(X)$ then 
$g=f-\sum_{i=1}^kg_i\cdot f(x_i(y))\in LE(X)$ and $g(x_i(y))=0$ for all $i$. So, $l_y(g)=0$ and 
$l_y(f)=\sum_{i=1}^k\lambda_i(y)f(x_i(y))$, where $\lambda_i(y)=\varphi(g_i)(y)$. 
Note that each $\lambda_i(y)$ is a real number because so is $\varphi(g_i)(y)$. Hence, for every $f\in LE(X)$ we have $\varphi(f)(y)=\sum_{i=1}^k\lambda_i(y)f(x_i(y))$ with $\lambda_i(y)=\varphi(g_i)(y)$. This implies that $\supp(y)\subset K$. Indeed, if $z\in\overline X\backslash K$ take a neighborhood $U$ of $z$ in $\overline X$ such that $U\cap K=\varnothing$. Then for every $f\in LE(X)$ with $f(X\backslash U)=0$ we have 
$\varphi(f)(y)=0$ since $f(K)=0$, so $z\not\in\supp(y)$.

Finally, let show that $\lambda_j(y)=0$ for every $x_j(y)\in K\backslash\supp(y)$. Take a function $f\in LE(X)$ with $f(x_j(y))=1$ and $f(x)=0$ for all $x\in U$, where $U$ is a neighborhood of $K\backslash\{x_j(y)\}$ such that $x_j(y)\not\in cl(U)$. Then, by condition $(a)$, $\varphi(f)(y)=0$. On the other hand,
$\varphi(f)(y)=\sum_{i=1}^k\lambda_i(y)f(x_i(y))=\lambda_j(y)$. So, $\lambda_j(y)=0$. Therefore, $\varphi(f)(y)=\sum\{\lambda_i(y)f(x_i(y)):x_i(y)\in\supp(y)\}$ for every $f\in LE(X)$.

$(c)$ Let $x_0\in supp(y_0)\cap U$, where $y_0\in\overline Y$ and $U\subset\overline X$ is open. Take a neighborhood $W$ of $x_0$ in $\overline X$ with
$cl(W)\subset U$. Then there is $f\in LE(X)$ such that $f(X\backslash W)=0$ and $\overline{\varphi(f)}(y_0)\neq 0$.
Striving for a contradiction, we can find a net $\{y_\alpha\}\subset\overline Y$ converging to $y_0$ such that $supp(y_\alpha)\cap U=\varnothing$ for every $\alpha$.
Hence, $\overline X\backslash cl(W)$ is a neighborhood of each $supp(y_\alpha)$. Since $f(X\backslash cl(W))=0$,  $\overline{\varphi(f)}(y_\alpha)=0$ for all $\alpha$. Therefore, 
$\lim_\alpha\overline{\varphi(f)}(y_\alpha)=\overline{\varphi(f)}(y_0)=0$, a contradiction.

$(d)$ Obviously for every $p,k$ we have $\overline Y_{p,k}=\overline Y_{k}\cap\widetilde Y_{p}$, where $\overline Y_{k}=\{y\in\overline Y:|\supp(y)|\leq k\}$ and $\widetilde Y_{p}=\{y\in\overline Y:a(y)\leq p\}$. Since the support map $\supp:\overline Y\rightsquigarrow\overline X$ is lower semi-continuous, the sets 
$\overline Y_{k}$ are closed. Indeed, if $y\not\in\overline Y_{k}$, then $\supp(y)$ contains at least $k+1$ different points $x_1(y),x_2(y),..,x_{k+1}(y)$ and we choose disjoint neighborhoods $O_i$ of $x_i(y)$. Then there exists a neighborhood $U$ of $y$ in $\overline Y$ such that $\supp(z)$ meets each $O_i$ for all $z\in U$. Hence, $U\subset\overline Y\backslash \overline Y_{k}$. 

Let show that the sets $\widetilde Y_{p}$ are also closed in $\overline Y$. Suppose that $a(y)>p$ for some $y\in\overline Y$. 
Then there exists $f\in LE(X)$ such that $|\overline f(x)|<1$ for all $x\in\supp(y)$ and 
$|\overline{\varphi(f)}(y)|>p$. Take
a neighborhood $U$ of $\supp(y)$ with $U\subset\{x\in\overline X:|\overline f(x)|<1\}$ and
choose another neighborhood $W$ of $\supp(y)$ such that $cl(W)\subset U$. Since $E(\overline X)$ contains a $QS$-algebra satisfying condition $(2.3)$,
 there is $h\in E(X)$ such that $\overline h(cl(W))=1$, 
$\overline h(\overline X\backslash U)=0$ and $h(x)\in [0,1]$ for all $x\in X$. Then  
$g=h\cdot f\in LE(X)$ and $|\overline g(x)|<1$ for all $x\in\overline X$. Moreover,
$g|(W\cap X)=f|(W\cap X)$. So, by condition $(a)$, $|\overline{\varphi(g)}(y)|=|\overline{\varphi(f)}(y)|>p$. Therefore, 
$V=\{z\in\overline Y:|\overline{\varphi(g)}(z)|>p\}$ is a neighborhood of $y$ with $V\cap\widetilde Y_{p}=\varnothing$. Hence,  
all $\overline Y_{p,k}$ are closed in $\overline Y$.

It remains to show that $Y\subset\bigcup_{p,k\geq 1}\overline Y_{p,k}$. So, fix $y\in Y$. By condition $(b)$, there are finitely many points $x_i(y)\in X$, $i\leq k$, such that $\supp(y)=\{x_1(y),x_2(y),..,x_k(y)\}$.
Let show there is $p$ with $y\in\overline Y_{p,k}$. According to the proof of condition $(b)$, there are real numbers $\lambda_i(y)$, $i\leq k$, such that
$\varphi(f)(y)=\sum_{i=1}^k\lambda_i(y)f(x_i(y))$ for all $f\in LE(X)$.
It suffices to show that $a(y)=\sum_{i=1}^k|\lambda_i(y)|$. To this end, consider the functions $f_n\in E(X)$, $n>1$, defined by $f_n(x_i(y))=\varepsilon_i(1-1/n)$, where 
$\varepsilon_i=1$ if $\lambda_i(y)>0$ and $\varepsilon_i=-1$ if $\lambda_i(y)<0$ (such functions $f_n$ exist because $E(X)$ is a $QS$-algebra). Clearly, $|f_n(x_i(y))|<1$ for all $i,n$ and 
$\lim_n\varphi(f_n)(y)=\sum_{i=1}^k|\lambda_i(y)|$. Hence $\sum_{i=1}^k|\lambda_i(y)|\leq a(y)$. The reverse inequality $a(y)\leq \sum_{i=1}^k|\lambda_i(y)|$ follows from $\varphi(f)(y)=\sum_{i=1}^k\lambda_i(y)f(x_i(y))$, $f\in LE(X)$. Therefore, $y\in\overline Y_{p,k}$ with $p\geq\sum_{i=1}^k|\lambda_i(y)|$.  

$(e)$ Let $y\in cl(\overline Y_{p,k}\cap Y)\backslash\overline Y_{k-1}$, so $\supp(y)=\{x_1(y),..,x_k(y)\}$ consists of $k$ points from $\overline X$. 
Take functions $g_i\in E(X)$ such that $\overline g_i(V_i)=1$ and $\overline g_i|\overline X\backslash U_i=0$, where $U_i, V_i$ are neighborhoods of $x_i(y)$ with $cl(V_i)\subset U_i$ and $U_i\cap U_j=\varnothing$ for $i\neq j$. 
(such functions exist because $E(\overline X)$ contains a $QS$-algebra satisfying condition $(2.3)$). Since 
$y\in cl(\overline Y_{p,k}\cap Y)\backslash\overline Y_{k-1}$ and the support function is lower semi-continuous, there is a net $\{y_\alpha\}\subset (\overline Y_{p,k}\cap Y)\backslash\overline Y_{k-1}$ converging to $y$ such that $\supp(y_\alpha)=\{x_1(y_\alpha),..,x_k(y_\alpha)\}$ with $x_i(y_\alpha)\in V_i\cap X$. Then, according to the proof of condition $(b)$, for every $\alpha$ the numbers $\lambda_i(y_\alpha)=\varphi(g_i)(y_\alpha)$ are finite and $\varphi(f)(y_\alpha)=\sum_{i=1}^k\lambda_i(y_\alpha)f(x_i(y_\alpha))$ for all $f\in LE(X)$. Moreover, $\sum_{i=1}^k|\lambda_i(y_\alpha)|=a(y_\alpha)\leq p$, 
$\lim_\alpha\varphi(f)(y_\alpha)=\overline{\varphi(f)}(y)$ and $\lim_\alpha\varphi(g_i)(y_\alpha)=\overline{\varphi(g_i)}(y)$ for all $i$. 
So, $\sum_{i=1}^k|\overline{\varphi(g_i)}(y)|\leq p$.
On the other hand, because of the lower semi-continuity of the support map and the fact that $\supp(y)$ and $\supp(y_\alpha)$ have the same number of elements, we can assume that $\lim_\alpha x_i(y_\alpha)=x_i(y)$ for all $i$. Consequently, $\lim_\alpha f(x_i(y_\alpha))=\overline f(x_i(y))$ and
we have  $\overline{\varphi(f)}(y)=\sum_{i=1}^k\overline{\varphi(g_i)}(y)\overline f(x_i(y))$ such that $\sum_{i=1}^k|\overline{\varphi(g_i)}(y)|\leq p$. 
Therefore, the numbers $\lambda_i(y)=\overline{\varphi(g_i)}(y)$ are finite and satisfy the equality $\overline{\varphi(f)}(y)=\sum_{i=1}^k\lambda_i(y)\overline f(x_i(y))$ for all $f\in LE(X)$. The last equality easily imply that 
$\lambda_i(y)=\overline{\varphi(g_i)}(y)$ for any function $g_i\in E(X)$ with $\overline g_i(x_i(y))=1$ and $\overline g_i(x_j(y))=0$ for all $j\neq j$. 
So, for every such $g_i$ and $z\in A(y)$ we have $\lambda_i(z)=\overline{\varphi(g_i)}(z)$ with $|\overline{\varphi(g_i)}(z)|\leq p$. This implies continuity of each $\lambda_i$ on the set $A(y)$. One can also show that $\sum_{i=1}^k|\lambda_i(y)|=a(y)$, see the arguments from the proof of $(d)$.
\end{proof}
For any space $X$ and an integer $k$ let $[X]^k$ denote the space of all $k$-points subsets of $\overline X$ endowed with the Vietoris topology. It is well know that if $X$ is metrizable, then the Hausdorff's metric on $[X]^k$ generates the Vietoris topology of $[X]^k$, see for example \cite{mi}.  Recall that if $(X,d)$ is a metric space and $A,B$ are two subsets of $X$, the Hausdorff's distance between $A$ and $B$ is $d_H(A,B)=\max\{h(A,B), h(B,A)\}$, where $h(A,B)=\sup\{d(a,B):a\in A\}$.
\begin{lem}
If $X$ is a separable metrizable space, then $\dim [X]^k\leq k\cdot\dim X$ for every $k$. 
\end{lem}
\begin{proof}
 Choose a countable base 
$\mathcal B$ for $X$. For every $k$-tuple $(U_1,U_2,..,U_k)$ of elements of $\mathcal B$ with pairwise disjoint closures let $$W(U_1,U_2,..,U_k)=\{\{x_1,x_2,..,x_k\}\in [X]^k:x_i\in cl(U_i), i=1,2,..,k\}.$$ Every $W(U_1,U_2,..,U_k)$ is homeomorphic to the product
$\prod_{i=1}^kcl(U_i)$. Indeed, the map $h:X^k\to [X]^{\leq k}$, defined by $h((x_1,x_2,..,x_k))=\{x_1,x_2,..,x_k\}$. According to \cite{mi}, $h$ is a continuous and closed surjection. Obviously, the restriction $h|\prod_{i=1}^kcl(U_i)$ is a closed bijection onto $W(U_1,U_2,..,U_k)$. So, it is a homeomorphism.
Because $\dim cl(U_i)\leq\dim X\leq m$, by \cite[Theorem 3.4.9]{en}, $\dim W(U_1,U_2,..,U_k)\leq k\cdot m$. Moreover, the family of all $W(U_1,U_2,..,U_k)$ is countable and covers $[X]^k$. Hence, by the Countable Sum Theorem, $\dim [X]^k\leq k\cdot\dim X$. 
 \end{proof}
\begin{pro}
Suppose in the hypotheses of Proposition $2.1$ we have the following additional conditions:
\begin{itemize}
\item Both $\overline X$ and $\overline Y$ are metric compactifications of $X$ and $Y$, respectively;
\item $|\supp(y)|\leq m$ for every $y\in Y$;
\end{itemize}
Then $\dim Y\leq m\cdot\dim\overline X$. 
\end{pro}
\begin{proof}
Since the support map $\overline Y\rightsquigarrow X$ is lower semi-continuous (see proposition 2.1(c)), $|\supp(y)|\leq m$ for all $y\in\overline Y$. 
By Proposition 2.1, the sets $\overline Y_{p,k}=\{y\in\overline Y:|\supp(y)|\leq k{~}\hbox{and}{~}a(y)\leq p\}$ and $\overline Y_k=\{y\in\overline Y:|\supp(y)|\leq k\}$ are closed in $\overline Y$. Because $|\supp(y)|\leq m$, $\overline Y=\overline Y_{m}$ and $Y\subset\bigcup\{\overline Y_{p,k}:1\leq k\leq m, p\geq 1\}$. For every $p\geq 1$ and $k\geq 2$ we define 
$$M(p,1)=\overline Y_{p,1}{~}\hbox{and}{~}M(p,k)=\overline Y_{p,k}\backslash\overline Y_{p,k-1}.$$
Some $M(p,k)$ could be empty but $Y\subset\bigcup\{M(p,k):1\leq k\leq m, p\geq 1\}$. Since all $\overline Y_{p,k}$ are closed in $\overline Y$, each
$M(p,k)$ is the union of countably many compact sets $F_n'(p,k)$, $n\geq 1$, and let $F_n(p,k)=cl(Y\cap F_n'(p,k))$.   
Because $\supp(y)$ consists of $k$ different points for any $y\in F_n(p,k)$, we have a map $S_n(p,k):F_n(p,k)\to [\overline X]^k$, 
$S_n(p,k)=\supp(y)$. 
According to Proposition 2.1(c), each $S_n(p,k)$ is continuous.
 Everywhere below for every $y\in F_n(p,k)$ denote by $A(y)$ the set $\{z\in F_n(p,k):S_n(p,k)(z)=S_n(p,k)(y)\}$, i.e, the fiber $(S_n(p,k))^{-1}(S_n(p,k)(y))$ generated by $y$. Since all $F_n(p,k)$ are compact and $S_n(p,k)$ are continuous, each $A(y)$ is a  compact subset of $F_n(p,k)$. 
\begin{claim}
Let $y\in Y\cap F_n(p,k)$ and $\supp(y)=\{x_1(y),x_2(y),..,x_k(y)\}$. Then  
there exist unique finite numbers $\lambda_i(y)$ such that $\overline{\varphi(f)}(z)=\sum_{i=1}^k\lambda_i(z)f(x_i(y))$ for all $f\in LE(X)$ and $z\in A(y)$.
 Moreover, $\sum_{i=1}^k\lambda_i(y)\leq p$ and the functions $\lambda_i$ are continuous on the set $A(y)$.
\end{claim}
Indeed, Claim 1 follows from Proposition 2.1(e). $\diamondsuit$

\begin{claim}
For every $y\in Y\cap F_n(p,k)$ there is a linear continuous map
$\varphi_y:C_p(\supp(y))\to C_p(A(y))$ such that $\varphi_y(C_p(\supp(y))$ is dense in  $C_p(A(y))$. 
\end{claim}
For every $z\in A(y)$ we have $\supp(z)=\supp(y)=\{x_1(y),..,x_k(y)\}\subset X$.
According to Claim 1, there are continuous real-valued functions $\lambda_i$ on $A(y)$ such that $\overline{\varphi(f)}(z)=\sum_{i=1}^k\lambda_i(z)f(x_i(y))$ for all $f\in LE(X)$ and $z\in A(y)$. So, 
for any  $h\in C(\supp(y))$ the formula
$\varphi_y(h)(z)=\sum_{i=1}^k\lambda_i(z)h(x_i(y))$  defines a continuous function $\varphi_y(h)\in C(A(y)$. 
Continuity of $\varphi_y$ with respect to the point-wise convergence topology is obvious. We claim that $\varphi_y(C_p(\supp(y))$ is dense in  $C_p(A(y))$. Indeed, take $\theta\in C_p(A(y))$ and its neighborhood $V\subset C_p(A(y))$. Then extend $\theta$ to a function $\overline\theta\in C(\overline Y)$.   
Because the set of real-valued elements of $E_p(\overline Y)$ is dense in $C_p(\overline Y)$, there is $\overline g\in E(\overline Y)$ with $\overline g|A(y)\in V$. 
Next, choose $f\in E(X)$ such that $\varphi(f)=g$ and let $h=f|\supp(y)$. Then  
$\overline g(z)=\overline{\varphi(f)}(z)=\sum_{i=1}^k\lambda_i(z)f(x_i(y))=\varphi_y(h)(z)$ for every $z\in A(y)$ and $\varphi_y(h)\in V$. Therefore, 
$\varphi_y(C_p(\supp(y))$ is dense in  $C_p(A(y))$.  $\diamondsuit$

\begin{claim}
For every $y\in Y\cap F_n(p,k)$ we have $|A(y)|\leq k$ and $\dim (Y\cap F_n(p,k))\leq k\cdot\dim\overline X$. 
\end{claim}
Since $|\supp(y)|=k$, $C_p(\supp(y))$ is isomorphic to $\mathbb R^k$. According to Claim 2, there is a linear continuous map from $C_p(\supp(y))$ to $C_p(A(y))$ whose image is dense in $C_p(A(y))$. It is well known that every finite-dimensional linear subspace of a locally convex vector topological space is closed. Hence, 
$\varphi_y(C_p(\supp(y)))=C_p(A(y))$. This implies $|A(y)|\leq k$. 

Because $S_n(p,k)$ is continuous and the set $F_n(p,k)$ is compact, $S_n(p,k)$ is a perfect map. So, the restriction 
$$\widetilde S_n(p,k)=S_n(p,k)|H_n(p,k):H_n(p,k)\to K_n(p,k)$$ is also a perfect map, where $K_n(p,k)=S_n(p,k)(F_n(p,k)\cap Y)$ and
$H_n(p,k)=S_n(p,k)^{-1}(K_n(p,k))$. Since $K_n(p,k)\subset [\overline X]^k$ and, by Lemma 2.2, $\dim [\overline X]^k\leq k\cdot\dim\overline X$, we have 
$\dim K_n(p,k)\leq k\cdot\dim\overline X$, see \cite[Theorem 4.1.3 and Theorem 4.1.7,]{en}. 
Moreover, all fibers of $\widetilde S_n(p,k)$ are of the form  $A(y)$ with $y\in Y\cap F_n(p,k)$ and, since they are finite,
$\dim H_n(p,k)\leq k\cdot\dim\overline X$, see \cite[Theorem 3.3.10]{en}. Finally, $Y\cap F_n(p,k)\subset H_n(p,k)$ implies $\dim (Y\cap F_n(p,k))\leq k\cdot\dim\overline X$. $\diamondsuit$

Now, we can complete the proof of Proposition 2.3. Since, $Y=\bigcup_{k\leq m}\{Y\cap F_n(p,k):n,p=1,2,..\}$ and 
$\dim(Y\cap F_n(p,k))\leq k\cdot\dim\overline X$ for all $n,p$ and $k\leq m$,  by the Countable Sum Theorem for $\dim$ we have  $\dim Y\leq m\cdot\dim\overline X$.
\end{proof}
\section{Some more preliminary results}
\begin{lem}
Let $X$ be a $k$-dimensional separable metric space and $Z_0$ be a metric compactification of $X$. Then for every countable subfamily $\Phi$ of 
$C(X)$ there is a metrizable compactification $Z_1$ of $X$ and a map $\theta^1_0: Z_1\to Z_0$ such that:
\begin{itemize}
\item $\dim Z_1=k$; 
\item every $f\in\Phi$ is continuously extendable to a map $\overline f:Z_1\to\overline{\mathbb R}$;
\item $\theta^1_0\circ j_1=j_0$, where each $j_i:X\hookrightarrow Z_i$ is the corresponding embedding.
\end{itemize}    
\end{lem}
\begin{proof}
For every $f\in\Phi$ denote by $Z_f$ the closure of $f(X)$ in $\overline{\mathbb R}$. Consider the diagonal product $h$ of the maps $j_0:X\hookrightarrow Z_0$ and $\triangle\{f:f\in\Phi\}$, where $j_0:X\hookrightarrow Z_0$ is the embedding of $X$. Then the closure $K$ of $h(X)$ in the product
$Z_0\times\prod_{f\in\Phi}Z_f$ is a compactification of $X$ such that every $f\in\Phi$, can be continuously extended to a map $\widetilde f:K\to\overline{\mathbb R}$. 
Let $\theta:\beta X\to K$ be the map witnessing that $\beta X$ is a compactification
of $X$ larger than $K$. Since $\dim\beta X=k$, by the Marde\v{s}i\'{c} factorization theorem \cite[Theorem 3.4.1]{en} there is a metrizable compactum $Z_1$ and maps $\nu:\beta X\to Z_1$ and $\eta:Z_1\to K$ such that $\dim Z_1=k$ and $\theta=\eta\circ\nu$. Evidently, $\nu|i(X)$ is a homeomorphism, where $i:X\hookrightarrow\beta X$ is the embedding. So, $Z_1$ is a compactification of $X$ and $j_1=\nu\circ i:X\to Z_1$ is an embedding. Because every $f\in\Phi$ is extendable to a function $\widetilde f:K\to\overline{\mathbb R}$, the composition $\overline f=\widetilde f\circ\eta$ is an extension of $f$ over $Z_1$. Obviously, the composition $\theta^1_0=(\pi|K)\circ\eta$ satisfies the equality $\theta^1_0\circ j_1=j_0$, where 
$\pi:Z_0\times\prod_{f\in\Phi}Z_f\to Z_0$ is the projection.
\end{proof}
Note that, if $\Phi$ in Lemma 3.1 consists of bounded functions, then all extensions $\overline f$ are real-valued functions on $Z_1$. 

For every space $X$ let $\mathcal F_X$ be the class of all maps from $X$ onto second countable spaces. For any two maps $h_1,h_2\in\mathcal F_X$ we write
$h_1\succ h_2$ if there exists a continuous map $\theta:h_1(X)\to h_2(X)$ with $h_2=\theta\circ h_1$. If $\Phi\subset D(X)$ we denote by $\triangle\Phi$ the diagonal product of all $f\in\Phi$. Clearly, $(\triangle\Phi)(X)$ is a subspace of the product $\prod\{\mathbb R_f:f\in\Phi\}$, and 
let $\pi_f:(\triangle\Phi)(X)\to\mathbb R_f$ be the projection.
Following \cite{gu}, we call a set $\Phi\subset D(X)$ {\em admissible} if the family
$\pi(\Phi)=\{\pi_f:f\in\Phi\}$ is a $QS$-algebra on $(\triangle\Phi)(X)$. 
We are using the following facts:
\begin{itemize}
\item[(3.1)] $\dim X\leq n$ if and only if for every $h\in\mathcal F_X$ there exists a $h_0\in\mathcal F_X$ such that $\dim h_0(X)\leq n$ and $h_0\succ h$ \cite{p}.
\item[(3.2)] If $\dim X\leq n$ and $\Phi\subset C(X)$ is countable, then there exists a countable admissible set $\Theta\subset C(X)$ containing $\Phi$ with $\dim(\triangle\Theta)(X)\leq n$. It follows from the proof of \cite[Lemma 2.2]{gu} that we can choose $\Theta\subset C^*(X)$ provided that
    $\Phi\subset C^*(X)$.
\item[(3.3)] For every countable $\Phi'\subset C(X)$ there is a countable admissible set $\Phi$ containing $\Phi'$ such that $(\triangle\Phi)(X)$ is homeomorphic to $(\triangle\Phi')(X)$. According to the proof of \cite[Lemma 2.4]{gu}, $\Phi$ could be taken to be a subset of $C^*(X)$ if $\Phi'\subset C^*(X)$. Moreover, if $\Phi'$ satisfies condition $(2.3)$, then $\pi(\Phi)$ also satisfies that condition.
\item[(3.4)] If $\{\Psi_n\}$ is an increasing sequence of admissible subsets of $C(X)$, then $\Psi=\bigcup_n\Psi_n$ is also admissible, see \cite[Lemma 2.5]{gu}.
    \end{itemize}
The next lemma is an analogue of \cite[Lemma 4.3]{ev}.
\begin{lem}
Let $X$ be a $k$-dimensional space and $\Psi_0\subset D(X)$ be a countable set. Then there is a countable admissible set $\Psi\subset D(X)$ containing $\Psi_0$
and a $k$-dimensional metrizable compactification $\overline X_\Psi$ of $X_\Psi=(\triangle\Psi)(X)$ having a countable base $\mathcal B$ such that:
\begin{itemize}
\item $\overline X_\Psi=(\triangle\overline\Psi)(\beta X)$ with $\overline\Psi=\{\overline h:h\in\Psi\}\subset C(\beta X,\overline{\mathbb R})$;
\item Each $\pi_h$, $h\in\Psi$, is extendable to a map $\overline\pi_h:\overline X_\Psi\to\overline{\mathbb R}$; 
\item $E(X_\Psi)=\{\pi_h:h\in\Psi\}\subset D(X_\Psi)$ is a countable $QS$-algebra on $X_\Psi$ and $E(\overline X_\Psi)=\{\overline\pi_{h}:h\in\Psi\}$ contains a countable $QS$-algebra on $\overline X_\Psi$ satisfying condition $(2.3)$;
\item For every finite open cover $\gamma$ of $\overline X_\Psi$ with elements from $\mathcal B$ the family $E(\overline X_\Psi)$ contains a partition of unity subordinated to $\gamma$.
\end{itemize}
\end{lem}
\begin{proof}
Let $\overline\Psi_0=\{\overline h\in C(\beta X,\overline{\mathbb R}):h\in\Psi_0\}$ and $Z_0=(\triangle\overline\Psi_0)(\beta X)$ and $X_0=(\triangle \Psi_0)(X)$. 
 Since $\dim\beta X\leq k$, by the Marde\v{s}i\'{c} factorization theorem \cite[Theorem 3.4.1]{en}, there is a metrizable compactum $Z_1$ and maps $\theta^1_0:Z_1\to Z_0$ and $\delta_1:\beta X\to Z_1$ 
 such that $\dim Z_1\leq k$ and $\theta^1_0\circ\delta_1=\triangle\overline\Psi_0$. 
 Now, fix a countable base $\mathcal B_1$ of $Z_1$, which is closed under finite intersections, and a countable $QS$-algebra $C_1$ on $Z_1$ satisfying condition $(2.3)$. 
 Let $\Omega_1$ be the set of all finite open covers of $Z_1$ consisting of elements of $\mathcal B_1$ and for every $\gamma\in\Omega_1$ take a partition of unity $\alpha_\gamma$ subordinated to $\gamma$. 
 Denote $E(Z_1)=\{\pi_h\circ\theta^1_0:h\in\overline\Psi_0\}\cup\{\alpha_\gamma:\gamma\in\Omega_1\}\cup C_1$ and $E_1=\{h|X_1:h\in E(Z_1)\}$, where $X_1=\delta_1(X)$. 
 Observe that $E_1$ is contained in $D(X_1)$ because all $\alpha_\gamma$ and $C_1$ consist of bounded functions and $\Psi_0\subset D(X)$.
  So,  there exists a $QS$-algebra
  $\Theta_1\subset D(X_1)$ on $X_1$ containing $E_1$.
 
We construct by induction a sequence metric compacta $\{Z_n\}_{n\geq 1}$ each having a countable base $\mathcal B_n$, continuous surjections $\delta_n:\beta X\to Z_n$ and $\theta^n_{n-1}:Z_n\to Z_{n-1}$, countable 
$QS$-algebras $\{\Theta_{n}\}_{n\geq 1}\subset D(X_{n})$ on $X_{n}=\delta_{n}(X)$, countable $QS$-algebras $C_n\subset C(Z_n)$ on $Z_n$, countable families 
$E(Z_n)\subset C(Z_n,\overline{\mathbb R})$ such that:
\begin{itemize}
\item[(1)] $\mathcal B_{n+1}$ contains all sets $(\theta^{n+1}_{n})^{-1}(U)$, $U\in\mathcal B_{n}$, and is closed under finite intersections;
\item[(2)] $C_{n+1}$ satisfies condition $(2.3)$ and $\{h\circ\theta^{n+1}_{n}:h\in C_{n}\}\subset C_{n+1}$;
\item[(3)] Every $Z_{n+1}$ is a metric compactum with $\dim Z_{n+1}\leq k$ such that for each $h\in\Theta_{n}$ the map $h\circ(\theta^{n+1}_{n}|X_{n+1})$ is extendable to a map $\widetilde h\in C(Z_{n+1},\overline{\mathbb R})$;
\item[(4)] $E(Z_{n+1})=\{\widetilde h:h\in\Theta_{n}\}\cup\{\alpha_\gamma:\gamma\in\Omega_{n+1}\}\cup C_{n+1}$, where $\Omega_{n+1}$ is the family of all finite open covers $\gamma$ of $Z_{n+1}$ with elements from $\mathcal B_{n+1}$ and $\alpha_\gamma$ is a partition of unity subordinated to $\gamma$;
\item[(5)] $\{h\circ(\theta^{n+1}_n|X_{n+1}):h\in\Theta_n\}\subset E_{n+1}=\{h|X_{n+1}:h\in E(Z_{n+1})\}\subset\Theta_{n+1}$;
\item[(6)] $\theta^{n+1}_{n}\circ\delta_{n+1}=\delta_{n}$.        
\end{itemize}
If the construction is performed for all $m\leq n$, let $E_n=\{h|X_n:\overline h\in E(Z_n)\}$ and $\Theta_{n}\subset D(X_n)$ be a countable $QS$-algebra on $X_n$ with $E_n\subset\Theta_{n}$. 
By Lemma 3.1, there exists a metrizable compactification $Z_{n+1}$ of $X_n$ and maps $\theta^{n+1}_n:Z_{n+1}\to Z_n$, $\delta_{n+1}:\beta X\to Z_{n+1}$ such that $\theta^{n+1}_{n}\circ\delta_{n+1}=\delta_{n}$,
$\dim Z_{n+1}\leq k$ and for each $h\in\Theta_{n}$ the map $h\circ(\theta^{n+1}_n|X_{n+1})$ is extendable to a map $\widetilde h:Z_{n+1}\to\overline{\mathbb R}$ (respectively, to a map  
$\widetilde h:Z_{n+1}\to\mathbb R$ in case $h$ is a bounded function), where $X_{n+1}=\delta_{n+1}(X)$.
Next, choose a base $\mathcal B_{n+1}$ of $Z_{n+1}$, a countable $QS$-algebra $C_{n+1}\subset C(Z_{n+1})$ and a countable family $E(Z_{n+1})$ satisfying condition $(1)-(5)$.

Let $Z$ be the limit space of the inverse sequence $S=\{Z_n,\theta^n_{n-1}:n\geq 2\}$. Since  $\dim Z_n\leq k$ for all $n$, $\dim Z\leq k$, see \cite[Theorem 3.4.11]{en}. Moreover, there is a map $\delta:\beta X\to Z$ such that $\theta_n\circ\delta=\delta_n$ for all $n$, where $\theta_n:Z\to Z_n$ are the projections in $S$. Let $\Psi_n=\{h\circ(\delta_n|X):h\in\Theta_n\}$ and $\Psi=\bigcup_{n\geq 1}\Psi_n$. Condition $(5)$ implies that $\{\Psi_n\}$ is increasing and, since $E_1$ contains $\{\pi_h\circ(\theta^1_0|X_1):h\in\overline\Psi_0\}$, $\Psi_0\subset\Psi_1$. Because $\{\pi_h:h\in\Psi_n\}=\Theta_n$, every $\Psi_n$ is admissible. Hence, $\Psi$ is also admissible and contains $\Psi_0$. Moreover, $\{\pi_h:h\in\Psi\}=\bigcup_{n\geq 1}\{h\circ(\theta_n|X_\Psi):h\in\Theta_n\}$ and, according to condition $(3)$, every $\pi_h$, $h\in\Psi$ is extendable to a map $\overline\pi_h\in C(Z,\overline{\mathbb R})$. Let $\overline\Psi=\{\overline\pi_h\circ\delta: h\in\Psi\}$. Because $C_n$ is a $QS$-algebra on $Z_n$, it separate the points and the closed subsets of $Z_n$. This fact implies that $Z$ is homeomorphic to $\overline X_\Psi=(\triangle\overline\Psi)(\beta X)$, so $X_\Psi=\delta(X)$.
  
So, we need to show that $E(Z)=\{\overline\pi_h:h\in\Psi\}$ contains a $QS$-algebra on $Z$ satisfying condition $(2.3)$ and $Z$ has a countable base $\mathcal B$ such that for every open cover $\gamma$ of $Z$ consisting of sets from $\mathcal B$ there is a partition of unity $\alpha_\gamma$ subordinated to $\gamma$ with $\alpha_\gamma\in E(Z)$. To this end let 
$C_n'=\{h\circ\theta_n:h\in C_n\}$, $n\geq 1$. Condition $(2)$ implies
the sequence $\{C'_n\}$ is increasing and let $C=\bigcup_n C'_n$.  
\begin{claim}
$C=\bigcup_n C'_n$ is a $QS$-algebra on $Z$ satisfying condition $(2.3)$.
\end{claim}
Since $\{C'_n\}$ is increasing, $C$ is closed under multiplications, additions and multiplications by rational numbers. So, to show it is a $QS$-algebra, we 
need to prove that for every point $z\in Z$ and its neighborhood $U\subset Z$ there is $h\in C$ such that $h(z)=1$ and $h(Z\backslash U)=0$. This will be done if we show $C$ satisfies condition $(2.3)$. To this end, take a compact set $K\subset Z$ and an open set $W\subset Z$ containing $K$. Then there exist $n$, a compact set $K_n\subset Z_n$ and open set $W_n\subset Z_n$ containing $K_n$ such that $K\subset\theta_n^{-1}(K_n)\subset\theta_n^{-1}(W_n)\subset W$. Since $C_n$ satisfies condition $(2.3)$, there is $h\in C_n$
with $h|K_n=1$, $h|(Z_n\backslash W_n)=0$ and $h(z)\in [0,1]$ for all $z\in Z_n$. Then 
$h'=\theta_n\circ h\in C$ is as required because $h':Z\to [0,1]$, $h'|K=1$ and $h'|(X\backslash W)=0$.  $\diamondsuit$

Recall that all finite intersections $U=\bigcap_{n=1}^m\theta_i^{-1}(U_n)$ with $U_n\in\mathcal B_n$ form a base $\mathcal B$ for $Z$. Because of the choice of all $\mathcal B_n$, see condition $(1)$, $\mathcal B$ consists of all sets of the form $U=\theta_n^{-1}(U_n)$ with $U_n\in\mathcal B_n$, $n\in\mathbb N$. 
\begin{claim}
For any finite open cover $\gamma$ of $Z$ consisting of sets from $\mathcal B$, the set $E(Z)$ contains a partition of unity subordinated to $\gamma$. 
\end{claim}
Indeed, for any such a cover $\gamma=\{U_1,..,U_m\}$ there is $n$ and a cover $\gamma_n=\{U_1^n,..,U_m^n\}\in\Omega_n$ such that $U_i=\theta_n^{-1}(U_i^n)$. So, there is a partition of unity 
$\alpha_{\gamma_n}=\{h_i^n:i=1,..,m\}$ subordinated to $\gamma_n$ with $\alpha_{\gamma_n}\subset E(Z_n)$.  
Then $\{h^n_i\circ\theta_n:i=1,...,m\}$ is a partition of unity subordinated to $\gamma$ and it is contained in $E(Z)$.
\end{proof}

\begin{lem}\cite[Lemma 4.2]{ev}
For every countable set $\overline\Phi'\subset C(\beta Y,\overline{\mathbb R})$ there is a countable set $\overline\Phi\subset C(\beta Y,\overline{\mathbb R})$ containing $\overline\Phi'$ such that 
the set of real-valued elements of $E_{\overline\Phi}=\{\pi_g: g\in\overline\Phi\}$ is dense in
$C_p((\triangle\overline\Phi)(\beta Y))$.
\end{lem}
\begin{lem}
Let $Y$ be a space and $D(Y)\subset C(Y)$. Denote by $D(\overline Y)$ the set of all extensions $\overline g\in C(\beta Y,\overline{\mathbb R})$, $g\in D(Y)$.
Then for every countable 
set $\overline\Phi_0\subset D(\overline Y)$ there is  a countable set $\overline\Phi\subset D(\overline Y)$ containing $\overline\Phi_0$ such that all real-valued elements of $E_{\overline\Phi}=\{\pi_{\overline g}: \overline g\in\overline\Phi\}$ is dense in
$C_p((\triangle\overline\Phi)(\beta Y))$ and the set $\Phi=\{\overline g|Y:\overline g\in\overline\Phi\}$ is admissible. 
\end{lem}
\begin{proof} Observe that $D(\overline Y)=C(\beta Y)$ if $D(Y)=C^*(Y)$.
\begin{claim}
For every countable $\Psi'\subset D(Y)$ there is a countable admissible set $\Psi\subset D(Y)$ containing $\Psi'$.
\end{claim}
We construct by induction an increasing sequence of countable sets $\Psi_n\subset D(Y)$ each containing $\Psi'$ and maps 
$\delta^{n+1}_n:(\triangle\Psi_{n+1})(Y)\to(\triangle\Psi_{n})(Y)$ such that: 
\begin{itemize}
\item Every space $Y_n=(\triangle\Psi_{n})(Y)$ has a countable base $\mathcal B_n$ with $(\delta^{n+1}_n)^{-1}(\mathcal B_n)\subset\mathcal B_{n+1}$;
\item $\Psi_{2n-1}$ is closed under addition, multiplication and multiplication by rational numbers;
\item For every pair $U,V\in\mathcal B_{2n-1}$ with $cl_{Y_n}(V)\subset U$ there exists a  fixed $f_{U,V}\in C^*(Y_{2n-1})$ such that 
$f_{U,V}(V)=1$, $f_{U,V}(Y_{2n-1}\backslash U)=0$ and $f_{U,V}\circ(\triangle\Psi_{2n-1})\in D(Y)$; 
\item $\Psi_{2n}=\Psi_{2n-1}\cup\{f_{U,V}\circ\triangle\Psi_{2n-1}:U,V\in\mathcal B_{2n-1}\}$.
\end{itemize} 
Since $D(Y)$ is an algebra, there is a countable $\Psi_1\subset D(Y)$ containing $\Psi'$ such that $\Psi_1$ is closed under addition, multiplication and multiplication by rational numbers. Then 
$\overline\Psi_1=\{\overline g:g\in\Psi_1\}\subset D(\overline Y)$ and $Y_1=(\triangle\Psi_{1})(Y)$ is a dense subset of 
$\overline Y_1=(\triangle\overline\Psi_{1})(\beta Y)$. We fix a countable base $\widetilde{\mathcal B}_1$ of $\overline Y_1$ and for every $\widetilde U,\widetilde V\in\widetilde{\mathcal B}_1$ with $cl(\widetilde V)\subset\widetilde U$ choose  
a function $\overline f_{\widetilde U,\widetilde V}\in C(\overline Y_1)$ such that  
$\overline f_{\widetilde U,\widetilde V}(\widetilde V)=1$ and $\overline f_{\widetilde U,\widetilde V}(\overline Y_1\backslash \widetilde U)=0$.
Denote $\mathcal B_1=\{\widetilde U\cap Y_1:\widetilde U\in\widetilde{\mathcal B}_1\}$ and $f_{U,V}=\overline f_{\widetilde U,\widetilde V}|Y_1$.

Observe that $\overline f_{U,V}\circ(\triangle\overline\Psi_{1})\in C(\beta Y)$ and $\big(\overline f_{U,V}\circ(\triangle\overline\Psi_{1})\big)|Y=f_{U,V}\circ(\triangle\Psi_{1}(Y))$. Then
each $f_{U,V}\circ(\triangle\Psi_{1}(Y))$ belongs to $D(Y)$. Next, define $\overline\Psi_2=\overline\Psi_1\cup\{\overline f_{\widetilde U,\widetilde V}\circ(\triangle\overline\Psi_{1}):\widetilde U,\widetilde V\in\widetilde{\mathcal B}_1\}$. Because $\overline \Psi_1\subset\overline\Psi_2$ there is natural map $\overline\eta^2_1:\overline Y_2=(\triangle\overline\Psi_2)(\beta Y)\to \overline Y_1$. Choose a base $\widetilde{\mathcal B}_2$ for $\overline Y_2$ which is closed under finite intersections and containing all 
$(\overline\eta^2_1)^{-1}(\widetilde U)$, $\widetilde U\in\widetilde{\mathcal B}_1$. Let $\Psi_2=\overline\Psi_2|Y_2$, $\mathcal B_2=\{\widetilde U\cap Y_2:\widetilde U\in\widetilde{\mathcal B}_2\}$ and $\delta^2_1=\eta^2_1|Y_2$.
Now, it is clear how to complete the construction of the sets $\Psi_n$. 

Let show that the countable set 
$\Psi=\bigcup_{n=1}^\infty\Psi_n$ is admissible. Since $\{\Psi_n\}$ is increasing and each $\Psi_{2n-1}$ is closed under addition, multiplication and 
multiplication by rationals, so is $\Psi$. That implies the set $E_\Psi=\{\pi_g:g\in\Psi\}$ is also closed under addition, multiplication and multiplication by rationals. Hence, we need to show that for every $y\in Y_\Psi=(\triangle\Psi)(Y)$ and its neighborhood $U\subset Y_\Psi$ there exists $g\in\Psi$ with
$\pi_g(y)=1$ and $\pi_g(Y_\Psi\backslash U)=0$. To this end, observe that $Y_{\overline\Psi}=(\triangle\overline\Psi)(\beta Y)$ is the limit space of the inverse sequence
$S_\Psi=\{\overline Y_n;\eta^{n+1}_n)\}$ and, according to our construction, the base of $Y_{\overline\Psi}$ consists of all sets of the form $\eta_n^{-1}(\widetilde U_n)$ with $\widetilde U_n\in\widetilde{\mathcal B}_n$, $n\geq 1$, where 
$\eta_n:Y_{\overline\Psi}\to\overline Y_n$ are the projections in $S_\Psi$. Therefore, we can assume that $U=\widetilde U\cap Y_\Psi$ and $\widetilde U=\eta_n^{-1}(\widetilde U_n)$ with $\widetilde U_n\in\widetilde{\mathcal B}_n$ for some $n$. Passing to a bigger integer, we can suppose that $n=2k-1$. Choose a neighborhood $V_n\in\mathcal B_n$ of $y_n=\eta_n(y)$ with $cl_{Y_n}(V_n)\subset U_n=\widetilde U_n\cap Y_n$. Then the function $f_{U_n,V_n}$ satisfies the equalities $f_{U_n,V_n}(V_n)=1$ and $f_{U_n,V_n}(Y_n\backslash U_n)=0$. Consequently, $g=f_{U_n,V_n}\circ(\triangle\Psi_{n})\in\Psi_{n+1}\subset\Psi$ and $\pi_g=f_{U_n,V_n}\circ\eta_n$. 
Hence, $\pi_g(y)=1$ and $\pi_g(Y_\Psi\backslash U)=0$. This complete the proof of Claim 6. $\diamondsuit$

Next step of our proof is to construct an increasing sequence of countable sets $\overline\Phi_n\subset C(\beta Y,\overline{\mathbb R})$ containing $\overline\Phi_0$ such that:
\begin{itemize}
\item The real-valued elements of each $E_{2n-1}=\{\pi_{\overline g}: \overline g\in\overline\Phi_{2n-1}\}$ is dense in
$C_p((\triangle\overline\Phi_{2n-1})(\beta Y))$;
\item Every $\Phi_{2n}=\{\overline g|Y:\overline g\in\overline\Phi_{2n}\}$ is an admissible set.
\end{itemize}
Let show how to construct the sets $\overline\Phi_{n}$. We choose a set $\overline\Phi_{1}$ according to Lemma 3.3. Then 
$\Phi_{1}=\{\overline g|Y:\overline g\in\overline\Phi_{1}\}$ is a countable subset of $D(Y)$. Hence, by Claim 6, there is a countable admissible set
$\Phi_2\subset D(Y)$ containing $\Phi_{1}$. Let $\overline\Phi_2=\{\overline g:g\in\Phi_2\}\subset C(\beta Y,\overline{\mathbb R})$ be the extension of $\Phi_2$. Next, apply Lemma 3.3 to find a countable set $\overline\Phi_3\subset C(\beta Y,\overline{\mathbb R})$ containing $\overline\Phi_2$ such that
the real-valued elements of $E_{3}=\{\pi_{\overline g}: \overline g\in\overline\Phi_{3}\}$ is dense in
$C_p((\triangle\overline\Phi_{3})(\beta Y))$ and denote $\Phi_3=\{\overline g|Y:\overline g\in\overline\Phi_{3}\}$. In this way, applying either Lemma 3.3 or Claim 6, we construct the sequence $\{\overline\Phi_n\}$. Let $\overline\Phi=\bigcup_{n=1}^\infty\overline\Phi_n$ and $\Phi=\bigcup_{n=1}^\infty\Phi_n$. Since $\Phi$ is the union of the increasing sequence $\{\Phi_{2n}\}$ consisting of admissible sets, $\Phi$ is also admissible, see $(3.4)$. 

So, it remains to show that the following claim:
\begin{claim}
The real-valued elements of $E_{\overline\Phi}=\{\pi_{\overline g}: \overline g\in\overline\Phi\}$ is dense in the space $C_p((\triangle\overline\Phi)(\beta Y))$. 
\end{claim}
To this end let $\overline Y_n=(\triangle\overline\Phi_n)(\beta Y)$ and $\overline Y_0=(\triangle\overline\Phi)(\beta Y)$. For every $n$ there are natural maps $\eta^{2n+1}_{2n-1}:\overline Y_{2n+1}\to \overline Y_{2n-1}$ since $\overline\Phi_{2n-1}\subset\overline\Phi_{2n+1}$. Then 
$\overline Y_0$ is the limit space of the inverse sequence $S_\Phi=\{\overline Y_{2n+1};\eta^{2n+1}_{2n-1}\}$.
Every projection $\eta_{2n-1}:\overline Y_0\to\overline Y_{2n-1}$ induces a continuous map 
$\eta_{2n-1}^*:C_p(\overline Y_{n-1})\to C_p(\overline Y_0)$ defined by $\eta_{2n-1}^*(h)=h\circ\eta_{2n-1}$. 
Because 
$E_{\overline\Phi}=\bigcup_n\eta_{2n-1}^*(E_{2n-1})$ (recall that $\overline\Phi=\bigcup_{n=1}^\infty\overline\Phi_{2n-1}$), it suffices to show that  
the set of real-valued functions from $\bigcup_{n=1}^\infty\eta_{2n-1}^*(E_{2n-1})$ is dense in $C_p(\overline Y_0)$. So, let 
$O=\{g\in C_p(\overline Y_0):|g(\overline y_i)-g_0(\overline y_i)|<\varepsilon_i, i=1,..,k\}$ be a neighborhood of some $g_0\in C_p(\overline Y_0)$, where 
all $\overline y_i$ are different points from $\overline Y_0$. Since the base of $\overline Y_0$ consists of all sets of the form $\eta_{2n-1}^{-1}(U_{2n-1})$ where $U_{2n-1}$ belongs to the base of $\overline Y_{2n-1}$, there is $m$ and different points $y_i\in\overline Y_{2m-1}$ such that $\eta_{2m-1}(\overline y_i)=y_i$ and $\eta_{2m-1}^{-1}(U_i)\subset g_0^{-1}(V_i)$, where $U_i$ are neighborhoods of $y_i$ in $\overline Y_{2m-1}$ and $V_i$ are the open intervals $(g_0(\overline y_i)-\varepsilon_i,g_0(\overline y_i)+\varepsilon_i)$. Then the set $W=\{h\in C_p(\overline Y_{2m-1}):h(y_i)\in V_i, i=1,2,..,k\}$  is open in $C_p(\overline Y_{2m-1})$. Since the set of real-valued functions from $E_{2n-1}$ is dense in $C_p(\overline Y_{2n-1})$, there is a real-valued function
$h_0\in E_{2m-1}$ with $h_0\in W$. Then $h_0\circ\eta_{2m-1}$ is a real-valued function from $\bigcup_{n=1}^\infty\eta_{2n-1}^*(E_{2n-1})\cap O$. So, the set of all real-valued functions from $E_{\overline\Phi}$ is dense in $C_p(\overline Y_0)$.
\end{proof}

\section{Proofs}
\textit{ Proof of Theorem $1.1$.}
Let $X$ and $Y$ be Tychonoff spaces with 
$\dim X=d$  and $T:D_p(X)\to D_p(Y)$ be a continuous linear
surjection with $|\supp(y)|\leq m$ for all $y\in Y$. 
The support of every $y\in\overline Y$ consists of all $x\in\beta X$ such that for every neighborhood $U$ of $x$ in $\beta X$ there is $f\in D(X)$ with $f(X\backslash U)=0$ and $\overline{T(f)}(y)\neq 0$, where $\overline{T(f)}:\beta Y\to\overline{\mathbb R}$ denotes the extension of $T(f)$ over $\beta Y$. Therefore, we can apply Proposition 2.1 with $\overline X=\beta X$, $LE(X)=E(X)=D(X)$, $\overline Y=\beta Y$, $LE(Y)=E(Y)=D(Y)$ and $\varphi=T$. The supports have all properties established in Proposition 2.1, in particular $\supp(y)\neq\varnothing$ for all $y\in\beta Y$ and the support function 
$\beta Y\rightsquigarrow\beta X$ is lower semi-continuous.  This implies that 
$|\supp(y)|\leq m$ for all $y\in\beta Y$ (recall that $|\supp(y)|\leq m$, $y\in Y$).

To show that $\dim Y\leq m\cdot d$, it suffices to prove that for every $h\in\mathcal F_{Y}$ there exists  $h_0\in\mathcal F_{Y}$ with $h_0\succ h$ and $\dim h_0(Y)\leq m\cdot d$, see condition $(3.1)$. We are going to find such $h_0$ and apply Proposition 2.3 to show that  $\dim h_0(Y)\leq m\cdot d$.
To this end, fix $h\in\mathcal F_{Y}$ and let $\overline h:\beta Y\to\overline{h(Y)}$ such that $\overline{h(Y)}$ is a metric compactification of $h(Y)$. 
We will construct by induction two increasing sequences of countable sets 
$\{\overline\Psi_n\}\subset C(\beta X,\overline{\mathbb R})$, $\{\overline\Phi_n\}\subset C(\beta Y,\overline{\mathbb R})$, metrizable compactifications $\overline X_n=(\triangle\overline\Psi_n)(\beta X)$ and $\overline Y_n=(\triangle\overline\Phi_n)(\beta Y)$ of
the spaces $X_n=(\triangle\Psi_n)(X)$ and $Y_n=(\triangle\Phi_n)(Y)$, where $\Psi_n=\overline\Psi_n|X$ and $\Phi_n=\overline\Phi_n|Y$, 
countable bases $\mathcal B(\overline X_n)$ and $\mathcal B(\overline Y_n)$ for $\overline X_n$ and 
$\overline Y_n$ and continuous surjections $\theta^{n+1}_n:\overline X_{n+1}\to\overline X_n$, $\delta^{n+1}_n:\overline Y_{n+1}\to\overline Y_n$
satisfying the following conditions (everywhere below, if $f\in C(X)$ then $\overline f:\beta X\to\overline{\mathbb R}$ denotes its extension):
\begin{itemize}
\item[(4.0)] $\Psi_n\subset D(X)$ and $\Phi_n\subset D(Y)$ for all $n$;
\item[(4.1)] $\triangle\overline\Phi_1\succ\overline h$, $\Phi_n\subset\{T(f):f\in\Psi_{n}\}\subset\Phi_{n+1}$ and $\Psi_n\subset\Psi_{n+1}$;
\item[(4.2)] For any $\overline f\in\overline\Psi_n$ the map $\pi_{\overline f}:\overline X_n\to\overline{\mathbb R}$ extends the map $\pi_f$, where $f=\overline f|X$;
\item[(4.3)] Every $\Psi_n$ is admissible, $\dim\overline X_n\leq d$, each $\mathcal B(\overline X_n)$ is closed under finite intersections and $(\theta_n^{n+1})^{-1}(U)\in\mathcal B(\overline X_{n+1})$, $U\in\mathcal B(\overline X_n)$;
\item[(4.4)] $E(\overline X_n)=\{\pi_{\overline f}:\overline f\in\overline\Psi_n\}$ contains a countable $QS$-algebra $C_n\subset C(\overline X_n)$ on $\overline X_n$ satisfying condition $(2.3)$ with $\{h\circ\theta^{n+1}_n:h\in C_{n}\}\subset C_{n+1}$;
\item[(4.5)] For every finite open cover $\gamma$ of $\overline X_n$, consisting of sets from $\mathcal B(\overline X_n)$, there exists a partition of unity $\alpha_\gamma$ subordinated to $\gamma$ with $\alpha_\gamma\subset E(\overline X_n)$;
\item[(4.6)] For any $\overline g\in\overline\Phi_n$ the map $\pi_{\overline g}:\overline Y_n\to\overline{\mathbb R}$ extends the map $\pi_g$, where $g=\overline g|Y$;
\item[(4.7)]  $\mathcal B(\overline Y_n)$ contains all sets $(\delta^{n}_{n-1})^{-1}(U)$, $U\in\mathcal B(\overline Y_{n-1})$, and is closed under finite intersections;
\item[(4.8)] Every $\Phi_n$ is admissible and the set of real-valued functions from $E(\overline Y_n)=\{\pi_{\overline g}: \overline g\in\overline\Phi_n\}$ is dense in $C_p(\overline Y_n)$.
\end{itemize}
Note that if $D(X)=C^*(X)$, then $\{\overline\Psi_n\}\subset C(\beta X)$ for all $n$. Similarly, $\{\overline\Phi_n\}\subset C(\overline Y)$ provided $D(Y)\subset C^*(Y)$. 

Since $\overline h(\beta Y)$ is a separable metrizable space, there is a countable set $\overline \Phi_1'\subset C(\beta Y)$ with $\overline h=\triangle\overline\Phi_1'$. 
By Lemma 3.4, there is a countable set $\overline\Phi_1\subset C(\beta Y,\overline{\mathbb R})$ containing $\overline{\Phi_1'}$ such that $\Phi_1=\{\overline g|Y:\overline g\in\overline\Phi_1\}$ is admissible and
$E(\overline Y_1)=\{\pi_{\overline g}:\overline g\in\overline\Phi_1\}$ contains a dense subset of $C_p(\overline Y_1)$, where 
$\overline Y_1=(\triangle\overline\Phi_1)(\beta Y)$. Let $Y_1=(\triangle\Phi_1)(Y)$ and  
 choose a countable set $\Psi_1'\subset D(X)$ with $T(\Psi_1')=\Phi_1$. Next, apply Lemma 3.2 to find  a countable admissible set $\Psi_1\subset D(X)$ containing 
$\Psi_1'$ and a metrizable compactification $\overline X_1$ of $X_1=\triangle\Psi_1(X)$ satisfying conditions $(4.2)-(4.5)$. Let $\overline\Psi_1=\{\overline f:f\in\Psi_1\}$.

Suppose the construction is done for all $i\leq n$. Let $\Phi_{n+1}'\subset D(Y)$ be a countable set containing $T(\Psi_n)$ and denote 
 $\overline{\Phi'}_{n+1}=\{\overline g:g\in\Phi'_{n+1}\}\subset C(\beta Y,\overline{\mathbb R})$. By Lemma 3.4, there is a countable set $\overline\Phi_{n+1}\subset C(\beta Y,\overline{\mathbb R})$ containing $\overline{\Phi'}_{n+1}$ such that the set
$\Phi_{n+1}=\{\overline g|Y:\overline g\in\overline\Phi_{n+1}\}$ is admissible and 
 $E(\overline Y_{n+1})=\{\pi_{\overline g}:\overline g\in\overline\Phi_{n+1}\}$ contains a dense subset of $C_p(\overline Y_{n+1})$, where 
$\overline Y_{n+1}=(\triangle\overline\Phi_{n+1})(\beta Y)$. Let  $Y_{n+1}=(\triangle\Phi_{n+1})(Y)$.
Note that $\Phi_n\subset\Phi_{n+1}$ because $\Phi_n\subset T(\Psi_n)$.
Next, choose a countable set $\Psi_{n+1}'\subset D(X)$ with $T(\Psi_{n+1}')=\Phi_{n+1}$ and apply Lemma 3.2 to find  a countable admissible set $\Psi_{n+1}\subset D(X)$ containing 
$\Psi_{n+1}'\cup\Psi_n$ and a metrizable compactification $\overline X_{n+1}$ of $X_{n+1}=(\triangle\Psi_{n+1})(X)$ satisfying conditions $(4.1)-(4.5)$.
Because $\overline\Psi_n\subset\overline\Psi_{n+1}$, $\triangle\overline\Psi_{n+1}\succ\triangle\overline\Psi_{n}$. Hence, there exists a map  
$\theta^{n+1}_n:\overline X_{n+1}\to\overline X_n$ defined by $\theta^{n+1}_n=\triangle\overline\Psi_{n}((\triangle\overline\Psi_{n+1})^{-1}(x))$. 
This completes the induction. 

Let $\overline\Psi=\bigcup_n\overline\Psi_n$, $\overline\Phi=\bigcup_n\overline\Phi_n$,
$\overline X_0=(\triangle\overline\Psi)(\beta X)$, $X_0=(\triangle\Psi)(X)$,
$\overline Y_0=(\triangle\overline\Phi)(\beta Y)$, $Y_0=(\triangle\Phi)(Y)$ and $\overline h_0=\triangle\overline\Phi$, where $\Psi=\bigcup_n\Psi_n$ and $\Phi=\bigcup_n\Phi_n$. Since $\triangle\overline\Phi_1\succ\overline h$, $h_0\succ h$. We are going to show that $\dim Y_0\leq m\cdot d$. 

Clearly, $\Phi=\{T(f): f\in\Psi\}$.
Since $\overline X_0$ is the limit space of the inverse sequence $S_X=\{\overline X_n,\theta^{n+1}_n\}$ and $\dim\overline X_n\leq d$ for all $n$, by \cite[Theorem 3.4.11]{en}, $\dim\overline X_0\leq d$. Moreover, since $\Psi$ is the union of the increasing sequence $\{\Psi_n\}$ of admissible sets, it is also admissible. Hence,
$E(X_0)=\{\pi_f:f\in\Psi\}$ is a countable $QS$-algebra on $X_0$ such that every $\pi_f$ is extendable to a continuous map $\pi_{\overline f}:\overline X_0\to\overline{\mathbb R}$ with $\overline f\in\overline\Psi$. Let $E(\overline X_0)=\{\pi_{\overline f}:\overline f\in\overline\Psi\}$ and
$C=\bigcup_n\{h\circ\theta_n:h\in C_n\}\subset E(\overline X_0)$. The same arguments as in the proof of Claim 4 from Lemma 3.2 show that $C$ is a $QS$-algebra on $\overline X_0$ satisfying condition $(2.3)$. Let $\mathcal B(\overline X)$ be the base of $\overline X_0$ generated by the bases $\mathcal B(\overline X_n)$. The arguments from the proof of Claim 5 from Lemma 3.2 also provide that 
 for every finite open cover $\gamma$ of $\overline X_0$, consisting of open sets from $\mathcal B(\overline X)$ there exists a partition of unity $\alpha_\gamma$ subordinated to $\gamma$ with $\alpha_\gamma\subset E(\overline X_0)$. 

The inclusions $\overline\Phi_n\subset\overline\Phi_{n+1}$ imply that $\triangle\overline\Phi_{n+1}\succ\triangle\overline\Phi_{n}$. So, for every $n$ there is a map $\delta^{n+1}_n:\overline Y_{n+1}\to\overline Y_n$ defined by $\delta^{n+1}_n(y)=\triangle\Phi_{n}((\triangle\Phi_{n+1})^{-1}(y))$.
Then $\overline Y_0$ is the limit of the inverse sequence $S_Y=\{\overline Y_n,\delta^{n+1}_n\}$. 
Let $E(Y_0)=\{\pi_g:g\in\Phi\}$ and $E(\overline Y_0)=\{\pi_{\overline g}:\overline g\in\overline\Phi\}$. Because each $\Phi_n$ is admissible and the sequence $\{\Phi_n\}$ is increasing, $\Phi$ is also admissible. This means that $E(Y_0)$ is a $QS$-algebra on $Y_0$ such that each $\pi_g$, $g\in\Phi$, is extendable to a map $\pi_{\overline g}:\overline Y_0\to\overline{\mathbb R}$. Moreover, the arguments from the proof of Claim 7 from Lemma 3.4 show that 
the set of real-valued elements of $E_p(\overline Y_0)$ is dense in $C_p(\overline Y_0)$.

We define a map 
$\varphi_0:E_p(X_0)\to E_p(Y_0)$ by  $\varphi_0(\pi_f)=\pi_{T(f)}$. Since $T$ is continuous and linear, one can show that $\varphi_0$ is semi-linear and continuous. Moreover, because $T(\Psi)=\Phi$, $\varphi_0$ is surjective.
We are going to prove that $\varphi_0$ has a continuous extension over the linear hull $LE_p(X_0)$. For every $f\in LE(X_0)$ let $f^*=f\circ(\triangle\Psi)\in C(X)$. Evidently, if  
$f=\sum_{i=1}^q\lambda_i\cdot f_i\in LE(X_0)$, $f_i\in E(X_0)$, then $f^*=\sum_{i=1}^q\lambda_i\cdot f_i^*$ with $f_i^*\in\Psi$. So, we have another description of the map $\varphi_0$: 
$\varphi_0(f)=\pi_{T(f^*)}$, $f\in E(X_0)$.
\begin{claim}
Let $f=\sum_{i=1}^q\lambda_i\cdot f_i\in LE(X_0)$ with $f_i\in E(X_0)$ for all $i$. If $y^*\in Y$ and $h_0(y^*)=y$, then
$T(f^*)(y^*)=\sum_{i=1}^q\lambda_i\cdot T(f_i^*)(y^*)=\sum_{i=1}^q\lambda_i\cdot\varphi_0(f_i)(y)$. 
\end{claim}
It suffices to show that $T(f^*)(y^*)=\varphi_0(f)(y)$ for all $f\in E(X_0)$. And that is true because $f^*\in\Psi$, so $T(f^*)\in\Phi$ and $T(f^*)(y^*)=\varphi_0(f)(y)$. $\diamondsuit$

We define $\varphi:LE(X_0)\to LE(Y_0)$ by 
$\varphi(\sum_{i=1}^q\lambda_i\cdot f_i)=\sum_{i=1}^q\lambda_i\cdot\varphi_0(f_i)$, where $\lambda_i\in\mathbb R$ and $f_i\in E(X_0)$. The continuity of $\varphi$ with respect to the point-wise topology is equivalent of the continuity of all linear functionals $l_y:LE_p(X_0)\to \mathbb R$ defined by $l_y(f)=\varphi(f)(y)$, $y\in Y_0$.
So, fix $y_0\in Y_0$ and $f_0=\sum_{i=1}^q\lambda_i\cdot f_i\in LE(X_0)$ with $f_i\in E(X_0)$ such that $l_{y_0}(f_0)\in V$ for some open interval $V\subset\mathbb R$. Then $f_0^*=\sum_{i=1}^q\lambda_i\cdot f_i^*$ and, by Claim 8 we have
$$T(f_0^*)(y_0^*)=\sum_{i=1}^q\lambda_i\cdot T(f_i^*)(y_0^*)=\sum_{i=1}^q\lambda_i\cdot\varphi_0(f_i)(y_0)=l_{y_0}(f_0),$$
where $y_0^*\in Y$ with $h_0(y_0^*)=y_0$.
Since $T$ is continuous, so is the linear functional $\mu$ on $D_p(X)$, defined by $\mu(f)=T(f)(y_0^*)$. So, there is a neighborhood $W^*=\{g\in D(X):|g(x_j^*)-f_0^*(x_j^*)|<\eta_j, j=1,2,..,p\}$ of $f_0^*$ in $D_p(X)$ such that $T(g)(y_0^*)\in V$ for all $g\in W^*$. Observe that 
$$f_0^*(x_j^*)=\sum_{i=1}^q\lambda_i\cdot f_i^*(x_j^*)=\sum_{i=1}^q\lambda_i\cdot f_i(x_j)=f_0(x_j),$$
where $x_j=(\triangle\Psi)(x_j^*)$. So, $W=\{f\in LE(X_0):|f(x_j)-f_0(x_j)|<\eta_j,j=1,2,..,p\}$ is a neighborhood of $f_0$ in $LE_p(X_0)$. If $g\in W$ and 
$g=\sum_{s=1}^t\lambda_s\cdot g_s$ for some $g_s\in E(X_0)$, then $g^*=\sum_{s=1}^t\lambda_s\cdot g_s^*\in W^*$ which means that
$T(g^*)(y_0^*)\in V$. Finally, according to Claim 8, $T(g^*)(y_0^*)=\sum_{s=1}^p\lambda_s\cdot\varphi(g_s)(y_0)=l_{y_0}(g)\in V$.
Thus, every $l_y$ is continuous and so is the map $\varphi:LE_p(X_0)\to LE_p(Y_0)$.

According to Proposition 2.1, the support of all $y\in\overline Y_0$ with respect to the linear map $\varphi$ are non-empty. 
\begin{claim}
For every $y\in\overline Y_0$ we have $|\supp(y)|\leq m$. 
\end{claim}
We fix $y\in\overline Y_0$ and let $y^*\in\beta Y$ with $\overline h_0(y^*)=y$. We claim that $\supp(y)\subset (\triangle\overline\Psi)(\supp(y^*))$. Indeed, let $x\in\overline X_0\backslash (\triangle\overline\Psi)(\supp(y^*))$ and $x^*\in\beta X$ with $(\triangle\overline\Psi)(x^*)=x$.  If $U\subset\overline X_0\backslash(\triangle\overline\Psi)(\supp(y^*))$ a 
neighborhood of $x$ in $\overline X_0$, then $(\triangle\overline\Psi)^{-1}(U)$ is a neighborhood of $x^*$ in $\beta X$ which is disjoint from $\supp(y^*)$. So, $\overline{T(f^*)}(y^*)=0$ for all $f^*\in D(X)$ with $f^*(X\backslash (\triangle\overline\Psi)^{-1}(U))=0$. Therefore, if $f\in LE(X_0)$ and $f(X_0\backslash U)=0$, then $\overline{\varphi(f)}(y)=\overline{T(f\circ\triangle\Psi)}(y^*)=0$. Hence, $x\not\in\supp(y)$, which means $\supp(y)\subset (\triangle\overline\Psi)(\supp(y^*))$.  $\diamondsuit$

Now, we can complete the proof of Theorem 1.1. Since the spaces $\overline X_0$, $\overline Y_0$ and the map $\varphi:LE(X_0)\to LE(Y_0)$ satisfy the hypotheses of Proposition 2.2, we have $\dim Y_0\leq m\cdot d$. This, according to the choice of the map $h_0$, implies $\dim Y\leq m\cdot d$.
$\Box$

\textit{ Proof of Corollary $1.2$.} Let $T:D_p(X)\to D_p(Y)$ be a continuous linear surjection such that both $X$ and $Y$ are normal spaces and $X$ is 
strongly countable-dimensional. Then there are countably many closed subsets $X_n\subset X$ with $\dim X_n\leq n$. For every two integers $k,n$ let 
$Y_{k,n}=\{y\in Y:\supp(y)\subset X_n{~}\hbox{and}{~}|\supp(y)|\leq k\}$. Since the support function is lower semi-continuous, see Proposition 2.1(c), the sets $Y_{k,n}$ are closed in $Y$, and since $X$ is normal, the projection maps $\pi_n:D_p(X)\to D_p(X_n)$ are linear continuous surjections. We define  a linear map $T_{k,n}:D_p(X_n)\to D_p(Y_{k,n})$, $T_{k,n}(h)=\{T(f)|Y_{k,n}:\pi_n(f)=h\}$. This definition is correct because $\supp(y)\subset X_n$ implies $T(f_1)(y)=T(f_2)(y)$ for any $y\in Y_{k,n}$ and $f_1,f_2\in D(X)$ with $\pi_n(f_1)=\pi_n(f_2)$. Moreover, $T_{k,n}$ is a continuous surjection since $Y$ is normal and $T$ is surjective. Finally, it follows from the definition of $\supp(y)$ that the support of every $y\in Y_{k,n}$ with respect to $T_{k,n}$ is $\supp(y)$. Therefore, by Theorem 1.1, $\dim Y_{k,n}\leq k\cdot\dim X_n\leq k\cdot n$. Since all supports $\supp(y)$, $y\in Y$, are finite subsets of $X$ (see Proposition 2.1), $Y=\bigcup\{Y_{k,n}:k,n\geq 1\}$. So, $Y$ is also strongly countable-dimensional.  $\Box$

\textit{ Proof of Theorem $1.3$.} Suppose $T:D_p(X)\to D_p(Y)$ is a continuous linear surjection with $\dim X=0$. Following the notations from the proof of Theorem 1.1, for every $h\in\mathcal F_Y$ we construct metric compacta $\overline X_0$, $\overline Y_0$, $QS$-algebras $E(X_0)$ and $E(Y_0)$, a map $h_0: Y\to Y_0$ and a linear surjection $\varphi:LE(X_0)\to LE(Y_0)$ satisfying the hypotheses of Proposition 2.3 such that $h_0\succ h$ and $\dim\overline X_0=0$. It's sufficient to proof that $\dim Y_0=0$. To this end, we consider the sets
$Y_0(m)=\{y\in Y_0:|\supp_\varphi(y)|\leq m\}$, $m\geq 1$, where $\supp_\varphi(y)$ is the support of $y$ with respect to $\varphi$. The arguments from the proof of Corollary 1.2 show that each $Y_0(m)$ is closed in $Y_0$ and there is a continuous linear surjection $\varphi_m:D_p(X_0)\to D_p(Y_0(m))$ defined by 
$\varphi_m(f)=\varphi(f)|Y_0(m)$. Since the support of every $y\in Y_0(m)$ with respect to $\varphi_m$ is the same as $\supp_\varphi(y)$,  
Proposition 2.3 implies $\dim Y_0(m)=0$. Since $Y_0=\bigcup_{m\geq 1}Y_0(m)$, the Countable Sum Theorem implies $\dim Y_0=0$. 
$\Box$

\textbf{Acknowledgments.} I would like to thank Prof. A. Leiderman for bringing Zakrzewski's paper to my attention. The author also thanks Prof. G. Dimov
for the the careful reading the first part of that paper and noticing some gaps.



\end{document}